\documentclass[11pt]{amsart}

\usepackage[colorlinks=true,citecolor=cyan,backref=page]{hyperref}

\usepackage[square,sort&compress,comma,numbers]{natbib}
\usepackage{nicefrac,xcolor,upref,version}

\usepackage{amscd,amsthm,amsmath,amssymb,amsfonts}

\usepackage{mathrsfs}

\newtheorem{theorem}{Theorem}[section]
\newtheorem{lemma}[theorem]{Lemma} 
\newtheorem{proposition}[theorem]{Proposition}

\newtheorem{corollary}[theorem]{Corollary}

\numberwithin{equation}{section}

\def\Q{{\mathbb {Q}}}
\def\C{{\mathbb {C}}}
\def\Z{{\mathbb Z}}  
\def\R{{\mathbb R}} 
\def\eps{{\varepsilon}}

\def\ts{{s'}} 
\def\tS{{S'}}

\def\eff{{\rm eff}}
\def\ineff{{\rm ineff}}
\def\resp{{\rm resp., }} 

\def\vv{{\rm v}} 
\def\pp{{\mathfrak p}}
\def\aa{{\mathfrak a}} 
\def\bb{{\mathfrak b}}

\def\beq{\begin{equation}}
\def\eeq{\end{equation}}

\def\cP{{\mathcal P}}

\begin{document}


\vskip 5mm

\title[On the difference between squares and integral $S$-units]
{On the difference between squares and integral $S$-units}

\author{Yann Bugeaud}
\address{I.R.M.A., UMR 7501, Universit\'e de Strasbourg
et CNRS, 7 rue Ren\'e Descartes, 67084 Strasbourg Cedex, France}
\address{Institut universitaire de France}
\email{bugeaud@math.unistra.fr}

\dedicatory{To Jan-Hendrik Evertse on his retirement}

\begin{abstract}
Let $q_1, \ldots , q_t$ be distinct prime numbers. 
Let $a_1, \ldots , a_t$ be nonnegative integers and $x$ a positive integer. 
We establish an effective lower bound for the greatest prime divisor of 
$|x^2 -  q_1^{a_1} \ldots q_t^{a_t}|$, which tends to infinity with 
the maximum of $x$, $a_1, \ldots , a_t$. 
\end{abstract}


\maketitle

\section{Introduction}\label{sec:1}

In his seminal paper \cite{Schi67},  Schinzel established fully explicit versions of two 
theorems of Gelfond \cite{Gel39,Gel60} on the Archimedean and 
the $p$-adic distances between integral powers of algebraic numbers (nowadays commonly referred to as linear forms in two logarithms, in the 
Archimedean and in the $p$-adic settings). He further gave many applications of his estimates, including the 
following result \cite[Theorem 9]{Schi67}.


\begin{theorem}[Schinzel, 1967] \label{SchiTh}
Let $m_1, \ldots, m_k$ be positive integers. If $x$ and $\ell_1, \ldots , \ell_k$ are positive integers
and $x^2 - m_1^{\ell_1} \ldots m_k^{\ell_k}$ is nonzero, then 
\begin{equation} \label{SchiMinor} 
| x^2 - m_1^{\ell_1} \ldots m_k^{\ell_k} | > \exp \bigl( c \, (\log \max \{x^2, m_1^{\ell_1} \ldots m_k^{\ell_k} \} )^{1/7} \bigr),
\end{equation} 
where $c$ is a positive effectively computable number depending only on $m_1, \ldots , m_k$. 
\end{theorem} 

By using more recent estimates for linear forms in two $p$-adic logarithms (established e.g. in \cite{BuLa96,Chim25}), the 
right hand side of \eqref{SchiMinor} can be replaced by a small positive power of 
$\max \{x^2, m_1^{\ell_1} \ldots m_k^{\ell_k} \} $, see \cite{BeBu12,Bu22}. 
Here, we go one step further, in the spirit of the papers \cite{GrVi13,BuEv17,BuEvGy18,Bu21}, where the authors bound from 
above the $S$-parts of various sequences of integers. 

Let $S = \{p_1, \ldots , p_s\}$ be a finite, non-empty set of distinct prime numbers.
For a nonzero integer $n$, write $n = b p_1^{a_1} \ldots p_s^{a_s}$, where 
$a_1, \ldots , a_s$ are non-negative integers and $b$ is an integer 
relatively prime to the product $p_1 \ldots p_s$. Then, we define the $S$-part $[n]_S$ 
of $n$ by 
$$
[n]_S := p_1^{a_1} \ldots p_s^{a_s}.
$$

In the sequel, 
we use the notation $\gg^{\eff}_{a, b, \ldots}$ and $\ll^{\eff}_{a, b, \ldots}$ (\resp $\gg^{\ineff}_{a, b, \ldots}$ 
and $\ll^{\ineff}_{a, b, \ldots}$) to indicate that the positive numerical constants implied by $\gg$ and $\ll$ 
can be given explicitly (\resp cannot be given explicitly from the known proofs) 
and depend at most on the parameters $a, b, \ldots$ 

Our main result is the following

\begin{theorem}  \label{main}
Let $S := \{ p_1, \ldots , p_s \}$ and $T := \{q_1, \ldots , q_t\}$ be disjoint sets of distinct prime numbers. 
For every nonzero integer $x$ coprime with $q_1 \ldots q_t$, every 
nonnegative integers $a_1, \ldots , a_t$, and every positive $\eps$, we have
$$
[x^2 - q_1^{a_1} \ldots q_t^{a_t}]_S \ll^{\ineff}_{S,T, \eps} |x^2 - q_1^{a_1} \ldots q_t^{a_t}|^{\frac{1}{2} + \eps} . 
$$
Furthermore, 
for every nonzero integer $x$ coprime with $q_1 \ldots q_t$ and every 
nonnegative integers $a_1, \ldots , a_t$, we have
\beq  \label{effecbound}
[x^2 - q_1^{a_1} \ldots q_t^{a_t}]_S \ll^{\eff}_{S,T}  |x^2 - q_1^{a_1} \ldots q_t^{a_t}|^{1 - \kappa},
\eeq
where 
\beq \label{P2}
\kappa = 
\bigl(c^s (\log \log P ) \bigl(  (\log p_1) \cdots (\log p_s) \bigr)^2 \bigr)^{-1} 
\eeq
and  $c$ is an effectively computable positive number  
depending only on $T$, and $P$ is the maximum of $p_1, \ldots , p_s$. 
\end{theorem}

The first assertion of Theorem \ref{main} is sharp. 
Indeed, let
$T := \{q_1, \ldots , q_t\}$ be a set of distinct prime numbers. 
Then, there are finite sets of primes $S$ with 
the smallest prime in $S$ arbitrarily large,
and, for each one of these sets $S$, infinitely many $(t+1)$-uples $(x, a_1, \ldots , a_t)$ 
of nonnegative integers such that
$$
[x^2 - q_1^{a_1} \ldots q_t^{a_t}]_S \gg_{S, T} |x^2 - q_1^{a_1} \ldots q_t^{a_t}|^{\frac{1}{2}}.
$$
This follows from Hensel's lemma, in a similar way as \cite[Theorem 2.1, (ii)]{BuEvGy18}. 

It would be interesting to solve completely Diophantine equations of the form 
$x^2 - q_1^{a_1} q_2^{a_2} = p^{u}$ (resp., $x^2 - q^{a}  = p_1^{u_1}  p_2^{u_2}$), for 
three given distinct prime numbers $q_1, q_2, p$ (resp., $q, p_1, p_2$). 
The more general equation
$$
x^2 - q_1^{a_1} \ldots q_t^{a_t} = y^n,
$$
where $q_1, \ldots , q_t$ are fixed prime numbers, 
has been investigated in \cite{Bu97,BenSik23,BMJS23} when $t=1$.

By applying Theorem \ref{main} with the set $S$ composed of the first $s$ prime numbers, 
we derive a lower bound for the greatest prime factor of $|x^2 - q_1^{a_1} \ldots q_t^{a_t}|$. 
For a nonzero integer $n$, let $P[n]$ denote the greatest prime factor of $|n|$
with the convention that $P[\pm 1] = 1$. 
For any positive real number $x$, we set $\log_* x = \max\{1, \log x\}$.  

\begin{corollary}  \label{gpdiv} 
Let $T := \{q_1, \ldots , q_t\}$ be a set of distinct prime numbers. 
Let $x$ be a nonzero integer coprime with $q_1 \ldots q_t$ and set
$$
X := \max\{ x^2, q_1^{a_1} \ldots q_t^{a_t}, 3\}.
$$
Then, we have
$$
P[ x^2 - q_1^{a_1} \ldots q_t^{a_t} ] \gg^{\eff}_T \log_* \log X
\, {\log_* \log_* \log X \over \log_* \log _*\log_* \log X}.
$$
\end{corollary}

To obtain this lower bound instead of the weaker estimate $\gg^{\eff}_T \log_* \log X$, we argue as Gy\H ory and Yu in the 
proof of their \cite[Theorem 2]{GyYu06}, see also \cite{EvGy15} and the remark at the end of Section \ref{sec:3}.


For completeness, we state the following result, whose proof follows the same lines as that 
of Theorem \ref{main}. 

\begin{theorem}  \label{effbound}
Let $T := \{q_1, \ldots , q_t\}$ be a set of distinct prime numbers. 
Let $m$ be a nonzero integer. 
There exists an effectively computable, positive real number $c$, depending only on $T$, 
such that, for every $(t+1)$-uple $(x, a_1, \ldots , a_t)$ 
with 
$$
x^2 - q_1^{a_1} \ldots q_t^{a_t} = m, 
$$
we have 
$$
| x | \le |2m|^c. 
$$
\end{theorem}

Let $S$, $T$, and $P$ be as in the statement of Theorem \ref{main}.  
Let $a_1, \ldots , a_t$ be nonnegative integers. 
For higher powers than squares, the analogue of the effective part of Theorem \ref{main} is a particular case 
of \cite[Theorem 2.5]{BuEvGy18}. Indeed, for an integer $d \ge 3$, we write 
$$
x^d -  q_1^{r_1} \ldots q_t^{r_t} \, \big(q_1^{\lfloor a_1 / d \rfloor} \ldots q_t^{\lfloor a_t / d \rfloor}\big)^d,
$$
where $r_i$ is the remainder in the Euclidean division of $a_i$ by $d$, for $i=1, \ldots , t$. Thus, we get $t^d$ integer binary forms 
$$
F_{\underline r} (X, Y) = X^d -  q_1^{r_1} \ldots q_t^{r_t} Y^d, \quad {\underline r} = (r_1, \ldots , r_t), 
$$
of degree $d$, to which \cite[Theorem 2.5]{BuEvGy18} applies. 
Thus, for every nonzero integer $x$ coprime with $q_1 \ldots q_t$ and any 
nonnegative integers $a_1, \ldots , a_t$, we have
$$
[x^d - q_1^{a_1} \ldots q_t^{a_t}]_S \ll^{\eff}_{S,T,d}  |x^d - q_1^{a_1} \ldots q_t^{a_t}|^{1 - \kappa},
$$
where 
\beq \label{Pd}
\kappa = 
\bigl(c^s \bigl(  P (\log p_1) \cdots (\log p_s) \bigr)^{d!} \bigr)^{-1} 
\eeq
and  $c$ is an effectively computable positive number  
depending only on $T$ and $d$. 
The dependence on $P$ is much better in \eqref{P2} than in \eqref{Pd}. 
The reason for this is that we do not use estimates for linear forms in $p_i$-adic logarithms in the 
proof of Theorem \ref{main} (namely, we use estimates for linear forms in $q_j$-adic logarithms), unlike in the 
proof of \cite[Theorem 2.5]{BuEvGy18}.

Furthermore, by arguing as in the proof of \cite[Theorem 2.4]{BuEvGy18}, 
we easily obtain that, 
for every integer $x$ and every positive $\eps$, we have
$$
[x^d - q_1^{a_1} \ldots q_t^{a_t}]_S \ll^{\ineff}_{S,T, \eps} |x^d - q_1^{a_1} \ldots q_t^{a_t}|^{\frac{1}{d} + \eps} . 
$$
Conversely, if $T := \{q_1, \ldots , q_t\}$ is a given set of distinct prime numbers, 
there are finite sets of primes $S$ with 
the smallest prime in $S$ arbitrarily large,
and for each one of these sets $S$ infinitely many $(t+1)$-uples $(x, a_1, \ldots , a_t)$ 
of nonnegative integers such that
$$
[x^d - q_1^{a_1} \ldots q_t^{a_t}]_S \gg_{S,T} |x^d - q_1^{a_1} \ldots q_t^{a_t}|^{\frac{1}{d}}.
$$

Finally, we point out that a result of similar strength than Theorem \ref{main} has been proved in 
\cite{Bu22} for the difference (or the sum) of two integral $T$-units. Namely, \cite[Theorem 3.8]{Bu22} asserts that 
if $x$ and $y$ are coprime integers whose prime divisors are in $T$ and $S$ is a finite set 
of prime numbers disjoint from $T$, then 
$$
[x + y]_S \ll^{\eff}_{S,T} |x+y|^{1 - \tau},
$$
where $\tau$ is an effectively computable positive number depending only on $S$ and $T$. 
By looking closely at the proof, we see that 
$$
\tau  = 
\bigl(c^s (\log \log P )  (\log p_1) \cdots (\log p_s)  \bigr)^{-1},
$$
where $c$ is an effectively computable positive number  
depending only on $T$, and $P$ is the maximum of $p_1, \ldots , p_s$.



\section{Auxiliary results}     \label{sec:2} 

In this section, we recall, in a simplified form, estimates from the theory of linear forms in logarithms. 
As usual, $h(\alpha)$ denotes the (logarithmic) Weil height of the algebraic number~$\alpha$
and we set $h_*(\alpha) = \max\{h(\alpha),1 \}$.

We begin with an immediate consequence of an estimate of Matveev \cite{Mat00}; see also \cite[Theorem 2.2]{Bu18b}.

\begin{theorem}   \label{lflog} 
Let $n \ge 2$ be an integer. 
Let $\alpha_1, \ldots, \alpha_n$ be nonzero algebraic numbers. 
Let $b_1, \ldots , b_n$ be integers with $b_n = \pm 1$.
Let $D$ be the degree over $\Q$ 
of the number field $\Q(\alpha_1, \ldots, \alpha_n)$. 
Set
$$
B = \max\{3, |b_1|, \ldots , |b_n|\}. 
$$ 
If
$$
\alpha_1^{b_1}  \ldots  \alpha_n^{b_n} \not= 1, 
$$ 
then there exists an 
effectively computable positive number $c_1$, depending only on $D$, 
such that 
\begin{equation*}  
\begin{split}
\log | \alpha_1^{b_1}  \ldots  \alpha_n^{b_n} - 1 |    
\ge - c_1^n  & \, h_* (\alpha_1) \ldots h_* (\alpha_n) \\
& \log_* \frac{B \max\{h_* (\alpha_1), \ldots , h_* (\alpha_{n-1}) \}}{h_* (\alpha_n)}.
\end{split}
\end{equation*}
\end{theorem} 

Let $p$ be a prime number. 
For a nonzero rational number $\alpha$, let $\vv_p (\alpha)$ denote the exponent of $p$ in the decomposition of $\alpha$ 
as a product of powers of prime numbers. 
More generally, if $\alpha$ is a nonzero algebraic number in a number field $K$, 
let $\vv_{\pp} (\alpha)$ denote the
exponent of $\pp$ in the decomposition of the
fractional ideal $\alpha O_K$ in a product of prime ideals and set 
$$
\vv_p (\alpha) = {\vv_{\pp} (\alpha) \over e_{\pp}}.
$$
This defines a valuation $\vv_p$ on $K$ which extends 
the $p$-adic valuation $\vv_p$ on $\Q$ normalized in such a way that $\vv_p (p) = 1$. 
We reproduce, in a simplified form, an estimate obtained by Yu \cite{Yu07}; see also \cite[Theorems 2.9 and 2.11]{Bu18b}. 

\begin{theorem} \label{Yu} 
Let $n \ge 2$ be an integer. 
Let $p$ be a prime number and $\alpha_1, \ldots, \alpha_n$ 
algebraic numbers in an algebraic number field of degree $D$.
Let $b_1, \ldots , b_n$ denote rational integers such that $b_n = \pm 1$ and $\alpha_1^{b_1} \ldots \alpha_n^{b_n}$
is not equal to $1$. 
Set
$$
B  =  \max\{3, |b_1|, \ldots , |b_n| \}.   
$$
There exists a positive effectively computable real number $c_2$,  depending only on $D$, such that 
\begin{equation*}  
\begin{split}
\vv_p (\alpha_1^{b_1} \ldots \alpha_n^{b_n} - 1)  < c_2^n \, & p^D \, 
h_* (\alpha_1) \ldots  h_* (\alpha_n) \\
&  \log_* \frac{B \max \{h_* (\alpha_1), \ldots , h_* (\alpha_{n-1}) \} }{h_* (\alpha_n)}.
\end{split}
\end{equation*}
\end{theorem} 


We may also refer to \cite[Theorem 3.2.8]{EvGy15}, which merges the statements of 
Theorems \ref{lflog} and \ref{Yu}.

In comparison with other lower bounds for linear forms in logarithms, 
the crucial point in Theorems \ref{lflog} and \ref{Yu} 
is the replacement of the factor 
$\log_* B$ by the factor 
$$
 \log_* \frac{B \max \{h_* (\alpha_1), \ldots , h_* (\alpha_{n-1}) \}  }{ h_* (\alpha_n)}, 
$$
which is much smaller than 
$\log B$ when $h_* (\alpha_n)$ is large. This is precisely the 
situation occurring in Section \ref{sec:3}; see also \cite{Bu23} for more explanations and further 
examples of Diophantine questions where this refined estimate appears to be crucial.

We state now a particular case of a classical lemma (see e.g. 
\cite[Proposition 4.3.12]{EvGy15}).

\begin{lemma}  \label{petitehauteur} 
Let $K$ be a real quadratic number field. Let $\eta$ denote its fundamental unit. 
For every nonzero algebraic integer $\alpha$ in $K$, there exists an integer $m$ such that 
$$
h(\eta^m \alpha) \le \frac{1}{2} \log | \alpha \sigma(\alpha)| + \frac{ \log \eta }{2}, 
$$
where $\sigma$ denotes the Galois embedding of $K$ different from the identity. 
\end{lemma}

\begin{proof}
Set $N =  | \alpha \sigma(\alpha)|$.  
Observe that $\sigma (\eta) = \pm \eta^{-1}$. 
There exist an algebraic integer $\delta$ in $K$ and an integer $m$ such that 
$$
\alpha = \delta \eta^m, \quad  \sigma(\alpha) = \pm \sigma(\delta)  \eta^{-m},  
$$
with
$$
N^{1/2}   \eta^{-1/2} < |\delta| \le N^{1/2}  \eta^{1/2}. 
$$
It suffices to observe that we get 
$$
| \sigma(\delta) | = \frac{N}{|\delta|} \le \eta^{1/2} N^{1/2},
$$
and the lemma follows from the definition of the height. 
\end{proof}

The proof of the ineffective part of Theorem \ref{main} rests on the $p$-adic Thue--Siegel--Roth theorem, 
formulated as in \cite[Theorem 6.2.3]{BG06}. In the next theorem, it is understood that, for a rational number $x/y$ and a prime number $p$, the 
expression $|x/y - \infty|_p$ means $|y/x|_p$ (this can be seen by applying a M\"obius transformation, see \cite[6.2.5]{BG06}). 

\begin{theorem}  \label{bgth}
Let $S_1$ be a finite set of finite places of $\Q$. For each $p$ in $S_1$, let $\beta_p$ be algebraic in $\Q_p$ or put $\beta_p = \infty$. 
Let $\beta_{\infty}$ be a real number. Let $\eps > 0$. Then, there are only finitely many rational numbers $x/y$ such that 
$$
\min \Bigl\{1, \Bigl|\frac{x}{y}-\beta_{\infty} \Bigr| \Bigr\} \, \prod_{p\in S_1} \, \min \Bigl\{1, \Bigl|\frac{x}{y}-\beta_{p} \Bigr|_p 
\Bigr\}  \le \max\{|x|, |y|\}^{- 2 - \eps}. 
$$
In particular, if $S_2$ is a subset of $S_1$, then, for any $\eps > 0$, there are only finitely many reduced rational numbers $x/y$ such that 
all the prime divisors of $y$ belong to $S_2$ and 
$$
\min \Bigl\{1, \Bigl|\frac{x}{y}-\beta_{\infty} \Bigr| \Bigr\} \, \prod_{p\in S_1 \setminus S_2} \, \min \Bigl\{1, \Bigl|\frac{x}{y}-\beta_{p} \Bigr|_p 
\Bigr\}  \le \max\{|x|, |y|\}^{- 1 - \eps}. 
$$
\end{theorem}

\begin{proof}
The second assertion follows from the first one by taking $\beta_p = \infty$ for every prime number $p$ in $S_2$. 
\end{proof}


\section{Proof of Theorem \ref{main}}   \label{sec:3}

\subsection{Proof of the ineffective estimate}

We argue as in the proof of \cite[Proposition 3.1]{BuEvGy18}.
We denote by $|\cdot |_{\infty}$ the ordinary absolute value,
and by $|\cdot |_p$ the $p$-adic absolute value normalized in such a way that $|p|_p=p^{-1}$ 
for a prime number $p$. 
Further, we set $\Q_{\infty}:=\R$ and $\overline{\Q}_{\infty}:=\C$.
The ineffective part of Theorem \ref{main} follows straightforwardly from the next result. 

\begin{proposition}\label{prop:3.1}
Let $T = \{q_1, \ldots , q_t\}$ be a finite, non-empty set of prime numbers. 
Let $S$ be a finite set of prime numbers, disjoint from $T$. 
Then
\[
\frac{|x^2 - q_1^{a_1} \ldots q_t^{a_t} |}{[x^2 - q_1^{a_1} \ldots q_t^{a_t}]_S}
\gg_{S,T,\varepsilon}^{\ineff} \max \{ x^2, q_1^{a_1} \ldots q_t^{a_t} \}^{\frac{1}{2}-\varepsilon}
\]
for every $\varepsilon >0$ and every positive integer $x$ coprime with $q_1 \ldots q_t$. 
\end{proposition}

\begin{proof} 
Set $y = q_1^{\lfloor a_1 / 2 \rfloor} \ldots q_t^{\lfloor a_t / 2 \rfloor}$. 
Our assumption implies that,
for each $p$ in $S\cup\{\infty\}$, there exists $\beta_p$ in $\overline{\Q}_p$ algebraic over $\Q$ such that 
$$
x^2 - q_1^{a_1} \ldots q_t^{a_t} = (x - \beta_p y) (x + \beta_p y).
$$
Then we have
\begin{eqnarray*}
&&\frac{|x^2 - q_1^{a_1} \ldots q_t^{a_t}|}{[x^2 - q_1^{a_1} \ldots q_t^{a_t}]_S\cdot \max \{x, y \}^2}
= 
\frac{ \prod_{p\in S\cup\{ \infty\}} |x^2 - q_1^{a_1} \ldots q_t^{a_t}|_p }{\max \{x, y \}^2}
\\[0.15cm]
&&\quad
\gg_{S,T} 
\prod_{p\in S\cup\{ \infty\}}\min \Bigl\{1, \Bigl|\frac{x}{y}-\beta_{p} \Bigr|_p, 
\Bigl|\frac{x}{y} + \beta_{p} \Bigr|_p\Bigr\}  \\ 
&&\quad
\gg_{S,T} 
\prod_{\substack{p\in S\cup\{ \infty\} \\ \beta_p \in \Q_p}} \min \Bigl\{1, \Bigl|\frac{x}{y}-\beta_{p} \Bigr|_p, 
\Bigl|\frac{x}{y} + \beta_{p} \Bigr|_p\Bigr\}.
\end{eqnarray*}
Since the prime factors of $y$ are in a finite set, the latter 
quantity is $\gg_{S,T,\varepsilon}^{\ineff} \max \{x, y \}^{-1-\varepsilon}$
for every $\varepsilon >0$, by Theorem \ref{bgth}.
\end{proof}

\subsection{Proof of the effective estimate}   \label{ssec:3.1}  

We adapt Schinzel's argument \cite{Schi67} and the proof of 
\cite[Theorem 1.2]{BeBu12}; see also \cite[Theorem 6.3]{Bu18b}. 
Throughout this subsection, all the numerical constants $c_1, c_2, \ldots$ 
and those implicit in $\ll$ and $\gg$ are effective and depend at most on $q_1, \ldots , q_t$. 

Define
$$
Q_1 = \prod_{1 \le i \le t, a_i \equiv 1 \!\!\!\!\! \pmod 2} \, q_i, \quad Q_2 = \sqrt{ \frac{q_1^{a_1} \ldots q_t^{a_t}}{Q_1} },
$$
thus, $q_1^{a_1} \ldots q_t^{a_t} = Q_1 Q_2^2$. Set 
$$
b := \frac{|x^2 - q_1^{a_1} \ldots q_t^{a_t}|}{[x^2 - q_1^{a_1} \ldots q_t^{a_t} ]_S} 
$$
and write
$$
[x^2 - q_1^{a_1} \ldots q_t^{a_t} ]_S = p_1^{2u_1 + v_1} \ldots p_s^{2 u_s + v_s}, \quad v_1, \ldots , v_s \in \{0, 1\}.
$$
Set
$$
A := \max\{|a_1|, \ldots , |a_t|, 2\}, \quad U := \max\{|u_1|, \ldots , |u_s|, 2\}, \quad X := \max\{x, Q_2, 2\}. 
$$
For later use, observe that 
$$
A \ll \log X, \quad U \ll \log X, \quad
|x^2 - Q_1 Q_2^2| \ll X^2. 
$$

Assume first that $Q_1 = 1$. Then, $a_1, \ldots , a_t$ are all even and we get
$$
|x - Q_2| \cdot |x + Q_2| = b p_1^{2u_1 + v_1} \ldots p_s^{2 u_s + v_s}.  
$$
Since the greatest common divisor of $x-Q_2$ and $x+Q_2$ is equal to $1$ or $2$, 
there exist an integer $h$ in $\{0, 1, \ldots , s\}$ (below, an empty product is always equal to $1$) 
and integers $b_1, b_2$ with $b_1 b_2 = 4 b$ such that (we reorder $p_1, \ldots , p_s$ if necessary) 
$$
x - Q_2 =   \frac{b_1}{2} p_1^{2u_1 + v_1} \ldots p_h^{2 u_h + v_h}, \quad
x + Q_2 =  \frac{b_2}{2} p_{h+1}^{2u_{h+1} + v_{h+1}} \ldots p_s^{2 u_s + v_s}.
$$
Consequently, 
\beq  \label{fllog1}
\frac{-2Q_2}{x + Q_2} 
= \frac{b_1}{b_2} p_1^{2u_1 + v_1} \ldots p_h^{2 u_h + v_h} p_{h+1}^{-2u_{h+1} - v_{h+1}} \ldots p_s^{-2 u_s - v_s} - 1. 
\eeq
Observe that 
$$
h \Bigl( \frac{b_1}{b_2} \Bigr)  \ll \log_* b.
$$
If $x > Q_2^2$, then it follows from Theorem \ref{lflog} applied to \eqref{fllog1} that 
\begin{align*}
\log X  & \le c_1^s (\log_* b) \prod_{i=1}^s (\log p_i) \cdot \Bigl( \log \frac{U}{\log_* b} \Bigr)  \\
& \le c_2^s (\log_* b) \prod_{i=1}^s (\log p_i) \cdot \Bigl( \log \frac{\log X}{\log_* b} \Bigr) , 
\end{align*}
and, setting
$$
\cP := (\log p_1) \cdots (\log p_s), \quad P := \max\{p_1, \ldots , p_s\},
$$
we get
$$
\frac{\log X}{\log_* b} \le c_2^s \cP  \Bigl( \log \frac{\log X}{\log_* b} \Bigr),
$$
thus, assuming that $\log X \ge 10 \log_* b$, we get 
$$
\frac{\log X}{\log_* b} \le 2 c_2^s \cP \log ( c_2^s \cP ) \le 2 c_2^s \cP s (\log \log P + \log c_2). 
$$
Consequently, we have established the lower bound 
$$
\max\{b, 3\} \ge X^{c_3^{-s} (\cP \log \log P)^{-1}} \ge |x^2 - Q_2^2|^{c_4^{-s} (\cP \log \log P)^{-1}}. 
$$
If $x \le Q_2^2$, then $A \gg U$ and 
we apply Theorem \ref{Yu} to \eqref{fllog1} to bound $\lfloor a_i / 2 \rfloor = \vv_{q_i} (Q_2)$ and we get
\begin{align*}
A & \le c_5^s (\log_* b) q_i \prod_{i=1}^s (\log p_i) \cdot  \Bigl( \log \frac{U}{\log_* b} \Bigr)   \\
& \le c_6^s (\log_* b) q_i \prod_{i=1}^s (\log p_i) \cdot \Bigl( \log \frac{A}{\log_* b} \Bigr) .
\end{align*}
This gives 
$$
\max\{b, 3\} \ge 2^{A c_7^{-s} (\cP \log \log P)^{-1}}  
\ge X^{c_8^{-s} (\cP \log \log P)^{-1}} \ge |x^2 - Q_2^2|^{c_{9}^{-s} (\cP \log \log P)^{-1}}.
$$

Assume now that $Q_1 \ge 2$.  Let $K$ denote the real quadratic field $\Q (\sqrt{Q_1})$ and $h$ its class number. 
The quadratic number
$$
\alpha = x - \sqrt{Q_1} Q_2 
$$
is an algebraic integer in $K$. 
Let $\pp_1, \ldots , \pp_\ts$ denote the prime ideals in $K$ that divide $p_1, \ldots , p_s$. 
Since at most two prime ideals in $K$ divide $p_i$ for $i=1, \ldots , s$, we have $\ts \le 2 s$. 
There is an ideal $\bb$ in $K$, coprime with $\pp_1 \ldots \pp_\ts$, 
and nonnegative integers $w_1, \ldots , w_\ts$ such that 
$$
(\alpha) = \bb \pp_1^{w_1} \ldots \pp_\ts^{w_\ts}. 
$$ 
For $i=1, \ldots , \ts$, let $z_i$ and $r_i$ denote, respectively, the quotient and the remainder in the Euclidean 
division of $w_i$ by $h$. Then, write
$$
(\alpha) = (\bb \pp_1^{r_1} \ldots \pp_\ts^{r_\ts} ) (\pp_1^h)^{z_1} \ldots (\pp_\ts^h)^{z_\ts}
$$
and observe that all the ideals in this equality are principal. 
Let $\eta$ denote the fundamental unit in $K$. By Lemma \ref{petitehauteur}, 
there is an integer $m$, an algebraic integer $\gamma$ 
in $K$ and, for $i=1, \ldots , \ts$, an algebraic integer $\pi_i$ generating the ideal $(\pp_i^h)$ such that 
\beq   \label{alphaequal}
\alpha = \pm \eta^m \gamma \pi_1^{z_1} \ldots \pi_\ts^{z_\ts}
\eeq
and
$$
h(\gamma) \le c_{10} \log_* b, \quad
h(\pi_i) \le c_{11} \log N \pp_i, \quad i=1, \ldots , \ts,
$$
where $N \pp$ is the norm of the ideal $\pp$. 
We need now to bound $|m|, z_1, \ldots , z_\ts$ from above. 

Since the algebraic integer $\alpha$ is a root of the quadratic polynomial
$$
Z^2 - 2 x Z + x^2 - q_1^{a_1} \ldots q_t^{a_t},
$$
its height $h(\alpha)$ satisfies
$$
h(\alpha)  \ll \log X. 
$$
For $i=1, \ldots , \ts$, the $\pp_i$-adic valuation of $\gamma$ is bounded by the class number 
$h$ and the $\pp_i$-adic valuations of the $\pi_j$'s are zero for $j \not= i$. 
It then follows from the definition of the height that 
$$
U \ll \max\{z_1, \ldots , z_{\ts} \} \ll h(\alpha). 
$$
Then, we deduce from \eqref{alphaequal} that 
$$
|m| \ll h(\eta^m) = h(\alpha \gamma^{-1} \pi_1^{-z_1} \ldots \pi_\ts^{-z_\ts} ) \ll (h_* (\gamma) + \log X). 
$$

Let $\sigma$ denote the complex embedding sending $\sqrt{Q_1}$ to $- \sqrt{Q_1}$. Then,
\beq \label{fllog2}
\Lambda :=
\frac{\alpha - \sigma(\alpha)}{\sigma (\alpha)} = 
\frac{-2 \sqrt{Q_1} Q_2}{x + \sqrt{Q_1} Q_2} 
= \frac{\gamma}{\sigma(\gamma)} \,  \Bigl( \frac{\eta}{\sigma(\eta)} \Bigr)^m \, \prod_{i=1}^\ts \, \Bigl( 
\frac{ \pi_i}{\sigma (\pi_i)} \Bigr)^{z_i} - 1. 
\eeq
If $x > Q_2^2$, then we apply Theorem \ref{lflog} to \eqref{fllog2} to get 
$$
\log X \ll - \log \Lambda   \le c_{12}^{s}  h(\eta) \, h_*(\gamma) 
\, \Bigl( \prod_{i=1}^\ts h_*(\pi_i) \Bigr)  \log \frac{h_* (\gamma) + \log X}{h_*(\gamma)}. 
$$
This gives the lower bound
$$
\max\{b, 3\} \ge X^{c_{13}^{-s} \cP^{-2} (\log \log P)^{-1}} \ge |x^2 - Q_1 Q_2^2|^{c_{14}^{-s}  \cP^{-2} (\log \log P)^{-1} }. 
$$
If $x \le Q_2^2$, then we apply Theorem \ref{Yu} to \eqref{fllog2}  
to bound $\lfloor a_i / 2 \rfloor = \vv_{q_i} (Q_2)$ and we get
$$
a_i \le c_{15}^{s} \, q_i  h(\eta) \, h_*(\gamma) \, \Bigl( \prod_{i=1}^\ts h_* (\pi_i) \Bigr) \log \frac{h_* (\gamma) + \log X}{h_*(\gamma)}. 
$$
Since $\log X \ll A$, this gives the lower bound 
$$
\max\{b, 3\} \ge X^{c_{16}^{-s}  \cP^{-2} (\log \log P)^{-1}} \ge |x^2 - Q_1 Q_2^2|^{c_{17}^{-s}  \cP^{-2} (\log \log P)^{-1}}, 
$$
and the proof of the effective part of Theorem \ref{main} is complete. 

We have established that, if $b=1$, then 
$$
X \le 3^{c_{18}^{s}  \cP^{2} (\log \log P)}. 
$$
By using the Prime Number Theorem and taking for $p_j$ the $j$-th prime number for $j = 1, \ldots , s$, 
we obtain Corollary \ref{gpdiv}. 

We conclude with a remark. Instead of several applications of Lemma \ref{petitehauteur}, we could have applied only once
\cite[Proposition 4.3.12]{EvGy15} to write $\alpha$ as a product of an algebraic number in $K$ of controlled height times an 
$S$-unit (the finite places in $S$ being $\pp_1, \ldots , \pp_\ts$). We would then have obtained \eqref{effecbound} with a 
different expression for $\kappa$ and the estimate 
$$
P[ x^2 - q_1^{a_1} \ldots q_t^{a_t} ] \gg^{\eff}_T \log_* \log X,
$$
which is weaker than our Corollary \ref{gpdiv}. 

\vskip 6mm

\section*{Acknowledgement}
The author is very thankful to the referee for a prompt and detailed report.


\begin{thebibliography}{99}




\bibitem{BeBu12}
M. A. Bennett and Y. Bugeaud,
{\it Effective results for restricted rational approximation to quadratic irrationals}, 
Acta Arith. 155 (2012), 259--269.


\bibitem{BenSik23}
M. A. Bennett and S. Siksek,
{\it Difference between perfect powers: prime power gaps}, 
Algebra Number Theory 17 (2023), 1789--1846.

\bibitem{BMJS23}
M. A. Bennett, Ph. Michaud-Jacobs, and S. Siksek,
{\it $\Q$-curves and the Lebesgue--Nagell equation}, 
J. Th\'eor. Nombres Bordeaux 35 (2023), 495--510. 


\bibitem{BG06}
E. Bombieri and W. Gubler. 
Heights in Diophantine geometry. 
New Mathematical Monographs, 4. Cambridge University Press, Cambridge, 2006.




\bibitem{Bu97}
Y. Bugeaud, 
{\it On the Diophantine equation $x^2 - p^m = \pm y^n$}, 
Acta Arith. 80 (1997), 213--223. 


\bibitem{Bu18b}
Y. Bugeaud, 
Linear forms in logarithms and applications. 
IRMA Lectures in Mathematics and Theoretical Physics 28, 
European Mathematical Society, Z\"urich, 2018.



\bibitem{Bu21}
Y. Bugeaud, 
{\it On the Zeckendorf representation of smooth numbers},
Moscow Math. J. 21 (2021), 31--42. 


\bibitem{Bu22}
Y. Bugeaud, 
{\it On effective approximation to quadratic numbers}, 
Acta Math. Spalatensia 2 (2022), 83--96. 

\bibitem{Bu23}
Y. Bugeaud, 
{\it $B'$},
Publ. Math. Debrecen 103 (2023), 499--533. 

\bibitem{BuEv17}
Y. Bugeaud and J.-H. Evertse, 
{\it $S$-parts of terms of integer linear recurrence sequences}, 
Mathematika 63 (2017), 840--851. 



\bibitem{BuEvGy18}
Y. Bugeaud, J.-H. Evertse, and K. Gy\H ory, 
{\it $S$-parts of values of of univariate polynomials, binary forms 
and decomposable forms at integral points}, 
Acta Arith. 184 (2018), 151--185. 



\bibitem{BuLa96}
Y. Bugeaud et M. Laurent,
{\it Minoration effective de la distance $p$-adique entre puissances de
nombres alg\'ebriques}, 
J. Number Theory 61 (1996), 311--342.


\bibitem{Chim25}
K. C. Chim, 
{\it Lower bounds for linear forms in two $p$-adic logarithms}, 
J. Number Theory 266 (2025), 295--349.

\bibitem{EvGy15}
J.-H. Evertse and K. Gy\H ory, 
Unit equations in Diophantine number theory. 
Cambridge Stud. Adv. Math. 146, Cambridge Univ. Press, 2015. 



\bibitem{Gel39}
 A. O. Gelfond, 
{\it Sur l'approximation du rapport des logarithmes de deux nombres alg\'ebriques au moyen de nombres alg\'ebriques}, 
Izv. Ross. Akad. Nauk Ser. Mat. 3 (1939), 509–518 (in Russian). 
 
 
 \bibitem{Gel60}
A. O. Gelfond. 
Transcendental and Algebraic Numbers, Dover Publ., New York, 1960. 




\bibitem{GrVi13}
S. Gross and A. Vincent, 
{\it On the factorization of $f(n)$ for $f(x)$ in $\Z[x]$}, 
Int. J. Number Theory 9 (2013), 1225--1236.



\bibitem{GyYu06}
K. Gy\H ory and K. Yu,
{\it Bounds for the solutions of $S$-unit equations and decomposable form equations},
Acta Arith. 123 (2006), 9--41. 



\bibitem{Mat00}
E. M. Matveev,
{\it An explicit lower bound for a homogeneous rational linear form in 
logarithms of algebraic numbers. II}, Izv.
Ross. Akad. Nauk Ser. Mat. 64 (2000), 125--180.


\bibitem{Schi67} 
A. Schinzel,
{\it On two theorems of Gelfond and some of their applications},
Acta Arith. 13 (1967), 177--236.


\bibitem{Yu07}
K. Yu,
{\it $p$--adic logarithmic forms and group varieties, III},  
Forum Math.  19  (2007),  187--280. 


\end{thebibliography}
\end{document}